\documentclass{amsart}
\usepackage{amsmath,amsfonts,amssymb,verbatim,graphicx,enumerate,float,amsthm}
\newtheorem{theo}{Theorem}[section]
\newtheorem{lemma}{Lemma}[section]
\newtheorem{corollary}{Corollary}[section]
\newtheorem{example}{Example}
\def \isnatural {\in\mathbb{N}}
\def \isreal {\in\mathbb{R}}
\def \iscomplex {\in\mathbb{C}}
\def \disc{\mathbb{D}}
\def \hplane{\mathbb{H}}
\newcommand{\tef}{transcendental entire function}
\newcommand{\ts}{transcendental singularity}
\newcommand{\tss}{transcendental singularities}
\newcommand{\dlts}{logarithmic singularity}
\newcommand{\dltss}{logarithmic singularities}
\newcommand{\dnlts}{direct non-logarithmic singularity}
\newcommand{\dnltss}{direct non-logarithmic singularities}
\newcommand{\its}{indirect singularity}
\newcommand{\itss}{indirect singularities}
\newcommand{\dts}{direct singularity}
\newcommand{\dtss}{direct singularities}
\newcommand\qfor{\quad\text{for }}
\begin{document}
\title[A new characterisation of the Eremenko-Lyubich Class]{A new characterisation of the Eremenko-Lyubich class}
\author{D. J. Sixsmith}
\address{Department of Mathematics and Statistics \\
	 The Open University \\
   Walton Hall\\
   Milton Keynes MK7 6AA\\
   UK}
\email{david.sixsmith@open.ac.uk}
%
%
%
%
\begin{abstract}
The Eremenko-Lyubich class of {\tef}s with a bounded set of singular values has been much studied. We give a new characterisation of this class of functions. We also give a new result regarding direct singularities which are not logarithmic.
\end{abstract}
\maketitle
%
%
%
%
\section{Introduction}
Throughout this paper we assume that $f$ is a {\tef}. We say that $f$ belongs to the Eremenko-Lyubich class, $\mathcal{B}$, if the set of finite critical and asymptotic values of $f$ is bounded. This set coincides with the set of singular values of the inverse function $f^{-1}$. The class $\mathcal{B}$ was introduced to complex dynamics in \cite{MR1196102}, and has been widely studied; see, for example, the papers \cite{MR2346947}, \cite{MR2570071}, \cite{MR1610785} and \cite{MR2753600} on the structure of the escaping set for functions in this class; \cite{MR2480096},  \cite{MR2587450} and \cite{MR1357062} on the dimensions of the Julia set and the escaping set; and \cite{Bish1} for an example of a function in this class with a wandering domain. Papers studying the value distribution of functions in this class include \cite{MR2081677}, \cite{MR1601855} and \cite{MR2478274}. 

An important property of functions in the Eremenko-Lyubich class is that they are expanding, in the following sense. Define $D_R = \{ z : |f(z)| > R \}$, for $R>0$. If $f\in\mathcal{B}$, then it follows easily from \cite[Lemma 1]{MR1196102} that there is a constant $R_0>0$ such that 
\begin{equation}
\label{expeq}
\left|z\frac{f'(z)}{f(z)}\right| \geq \frac{1}{4\pi}(\log|f(z)| - \log R_0), \qfor z \in D_{R_0}.
\end{equation}
This property has many applications in complex dynamics and value distribution theory; for example, it was used in \cite{MR1196102} to show that functions in the Eremenko-Lyubich class cannot have escaping Fatou components.

Define  
\begin{equation}
\label{def_of_s}
\eta_f = \lim_{R\rightarrow\infty} \inf_{z \in D_R} \left|z \frac{f'(z)}{f(z)}\right|.
\end{equation}
It follows from (\ref{expeq}) that if $f$ is a {\tef} in the Eremenko-Lyubich class, then $\eta_f~=~\infty$. The main result of this paper is the following, which shows that this property has a strong converse.
\begin{theo}
\label{theo1}
Suppose that $f$ is a {\tef}. Then, either $\eta_f = \infty$ and $f\in\mathcal{B}$, or $\eta_f = 0$ and $f\notin\mathcal{B}$.
\end{theo}
It is clear that if $f$ has an unbounded set of critical values, then $\eta_f=0$. Thus the proof of Theorem~\ref{theo1} requires detailed analysis of the behaviour of functions with an unbounded set of asymptotic values. Since every asymptotic value of $f$ gives rise to a {\ts} of $f^{-1}$, we need a number of results on singularities of the inverse function. In particular we require the following result on the density of {\tss} of a certain type, which may be of independent interest. Definitions of terms used in the statement of this theorem are given in Section~\ref{S1}.
\begin{theo}
\label{theo2}
Suppose that $f$ is a {\tef}, with a {\dnlts} with projection $a\in\widehat{\mathbb{C}}$. Then at least one of the following holds:
\begin{enumerate}[(i)]
\item $a$ is the limit of critical values of $f$;
\item every neighbourhood of this singularity contains a neighbourhood of another {\ts} of $f^{-1}$ that is either indirect or logarithmic, and whose projection is different from $a$.
\end{enumerate}
\end{theo}
We observe that Theorem~\ref{theo2} is complementary to the following result of Bergweiler and Eremenko \cite[Theorem 5]{MR2507243}, which has the same hypothesis although in this result the projection of the {\ts} must be finite.
\begin{theo}
\label{OtherBergErem}
Suppose that $f$ is a {\tef}, with a {\dnlts} with projection $a\iscomplex$. Then every neighbourhood of this singularity is also a neighbourhood of other {\dtss} of $f^{-1}$ with projection~$a$.
\end{theo}
Taken together, these results show that if $a\iscomplex$ is the projection of a {\dnlts} and is not the limit of critical values, then there is an infinite number of singularities both over $a$ and over points arbitrarily close to $a$. \\

%
%
We mention two examples of {\tef}s with {\dnltss} which illustrate some of the possibilities described above.
\begin{example}
\normalfont Heins \cite[p.435]{MR0094457} gave the example $f_1(z) = e^z \sin(e^z)$, which has precisely one {\dnlts} over $\infty$. Since the set of critical values of $f_1$ is unbounded, case $(i)$ of Theorem~\ref{theo2} holds for this function. This example also shows that Theorem~\ref{OtherBergErem} cannot be strengthened to $a\in\widehat{\mathbb{C}}$.
\end{example}
\begin{example}
\normalfont Herring \cite{MR1642181} gave the example $f_2(z) = \int_0^z \exp(-e^t) \ dt$. This function has no critical points. It follows from results in \cite{MR1642181} that $f_2$ has a {\dnlts} over $\infty$, every neighbourhood of which contains a left half-plane. It also follows that within each set $$A_k = \{ z : \operatorname{Re}(z) > 0, |\operatorname{Im}(z) - 2k\pi|\leq \pi/2\}, \qfor k\in\mathbb{Z},$$ there is a neighbourhood of a {\dlts} with projection $$\alpha_k = \alpha + 2k\pi i, \text{ where } \alpha\iscomplex \text{ is constant}.$$ Moreover, each neighbourhood of the {\dnlts} over $\infty$ contains neighbourhoods of these {\dltss}. Hence case $(ii)$ of Theorem~\ref{theo2} holds for $f_2$.
\end{example}
The structure of this paper is as follows. In Section~\ref{S1} we give details of Iversen's classification of singularities. We then prove Theorem~\ref{theo2} in Section~\ref{S1A}. Finally, in Section~\ref{S2}, we use Theorem~\ref{theo2} to prove Theorem~\ref{theo1}.
%
%
%
%
%
\section{Singularities of the inverse function}
\label{S1}
We write $\widehat{\mathbb{C}} = \mathbb{C}\cup\{\infty\}$, and use $B(w, r)$ to refer to the open disc around the point $w\iscomplex$, of radius $r$. We also write $B(\infty, r) = \{~z~:~|z|~>~r\}$. We write $\disc$ for the unit disc $B(0, 1)$, and $\disc^*$ for the punctured unit disc $B(0, 1)\backslash\{0\}$. We write $\hplane$ for the left half-plane $\{z : \operatorname{Re}(z) < 0\}$.

We recall Iversen's classification of singularities; see, for example, \cite{MR1344897}, \cite{MR2507243}, and \cite{Iversen}. Suppose that $f$ is a {\tef}, and suppose that $a\in\widehat{\mathbb{C}}$. For each $r > 0$, we can choose a component $U(r)$ of $f^{-1}(B(a, r))$ so that $r_1 < r_2$ implies that $U(r_1) \subset U(r_2)$. Then we have two possibilities:
\begin{enumerate}[(a)]
\item $\cap_{r>0} U(r)$ consists of a single point $w$, say, or
\item $\cap_{r>0} U(r) = \emptyset$.
\end{enumerate}

In the first case, if $f'(w) = 0$, then $w$ is a \itshape critical point \normalfont of $f$, $a$ is a \itshape critical value \normalfont of $f$, and we say that the singularity is \itshape algebraic\normalfont.

In the second case we say that the choice $r \mapsto U(r)$ defines a \itshape {\ts} \normalfont of $f^{-1}$, and we say that $a$ is the \itshape projection \normalfont of the {\ts} or equivalently that the {\ts} is \itshape over \normalfont $a$. Any of the sets $U(r)$ is called a \itshape neighbourhood \normalfont of the {\ts}. 

We say that a {\ts} over a point $a$ is \itshape direct \normalfont if there exists $r > 0$ such that $f(z) \ne a$, for $z \in U(r)$. Otherwise we call the {\ts} \itshape indirect\normalfont. We call a direct {\ts} over a point $a$ \itshape logarithmic \normalfont if, for some $r>0$, the restriction $f : U(r) \rightarrow B(a, r)\backslash\{a\}$ is a universal covering. If a {\ts} is direct but not logarithmic, we use the term \itshape direct non-logarithmic\normalfont.

We call a curve $\Gamma: (0, 1) \rightarrow \mathbb{C}$ an \itshape asymptotic curve \normalfont with \itshape asymptotic value \normalfont $a$ if, as $t\rightarrow 1$, we have both $\Gamma(t)\rightarrow\infty$ and $f(\Gamma(t))\rightarrow a$. Given a {\ts} over a point $a$ it is possible to construct an asymptotic curve with asymptotic value $a$, and \itshape vice versa\normalfont; see \cite[p.2]{MR1344897} for details.
%
%
%
%
%
\section{{\dnltss}}
\label{S1A}
In this section we prove Theorem~\ref{theo2}. We need the following \cite[Theorem~$4'$]{MR0094457}.
\begin{theo}
\label{Heins}
Suppose that $f$ is a {\tef}, $D\subset\mathbb{C}$ is a domain, and $W$ is a component of $f^{-1}(D)$. Then either $f_W$, the restriction of $f$ to $W$, has finite constant valence in $D$, or else there is at most one point of $D$ at which the valence of $f_W$ is finite.
\end{theo}
Here the \itshape valence \normalfont of a point $a\in D$ is the number of solutions of $f(z) = a$, for $z \in W$. It follows from Theorem~\ref{Heins} that there cannot be two distinct points $a, a' \in D$ such that $f(z) \in \{a, a'\}$ has no solutions, for $z \in W$. 

We also need the following result, and two straightforward corollaries of it. This seems to be well known, and follows quickly from results such as \cite[Example 4.2]{MR1326604}. See also \cite[Theorem 6.1.1]{MR2757285} for a detailed proof.
\begin{theo}
\label{Tcovering}
Suppose that $W\subset\mathbb{C}$ is a domain, and $g: W \to \disc^*$ is an unbranched covering map. Then exactly one of the following holds:
\begin{enumerate}[(i)]
\item there exists a conformal map $\phi: W \to \hplane$ such that $g = \exp \circ \ \phi$;
\item there exists a conformal map $\phi: W \to \disc^*$ such that $g = (\phi)^m$, for some $m\isnatural$.
\end{enumerate}
\end{theo}
%
%
The first corollary is similar to \cite[Theorem 6.2.2]{MR2757285}, and differs from that result in that it specifies the location of the neighbourhoods of the singularities, which is necessary for the proof of Theorem~\ref{theo2}. We give a proof for completeness.
\begin{corollary}
\label{limitcorollary}
Suppose that $f$ is a {\tef} with a {\ts} which is not logarithmic, over a point $a\in\widehat{\mathbb{C}}$. Then at least one of the following holds:
\begin{enumerate}[(i)]
\item $a$ is the limit of critical values of $f$;
\item every neighbourhood of this singularity contains a neighbourhood of another {\ts} of $f^{-1}$ whose projection is different from $a$.
\end{enumerate}
\end{corollary}
\begin{proof}
Suppose that, contrary to the conclusion of the corollary, we can choose a sufficiently small $r>0$ such that there are no critical points of $f$ in $$W=U(r)\backslash\{z:f(z)=a\},$$ and all {\tss} of $f^{-1}$ with a neighbourhood contained in $U(r)$ have projection $a$. The restriction of $f$ to $W$ is, therefore, an unbranched covering of $B(a, r)\backslash\{a\}$.

Let $h$ be a conformal map from $B(a, r)\backslash\{a\}$ to $\disc^*$. We apply Theorem~\ref{Tcovering} with $g = h \circ f$. If case $(i)$ of the theorem holds, then $W$ is simply connected, and the singularity is logarithmic, which is a contradiction. If case $(ii)$ of the theorem holds, then the conformal mapping $\phi$ has a punctured disc in the Riemann sphere as its domain, and at the puncture $\phi$ has a removable singularity. Hence, since $f=h^{-1} \circ (\phi)^m$, the singularity is algebraic; this is also a contradiction.
\end{proof}
The second corollary of Theorem~\ref{Tcovering} is similarly straightforward, and we omit the proof.
\begin{corollary}
\label{dltscorollary}
Suppose that $f$ is a {\tef} with a {\dlts} over a point $a\in\widehat{\mathbb{C}}$. Then there exist a neighbourhood of the singularity, $W=U(r)$, and conformal maps $\phi: W \to \hplane$ and $h : B(a,r) \to \disc^*$ such that $h \circ f = \exp \circ \ \phi$.
\end{corollary}
We now prove Theorem~\ref{theo2}.
\begin{proof}[Proof of Theorem~\ref{theo2}]
Suppose that $f$ has a {\dnlts} over a point $a~\in~\widehat{\mathbb{C}}$, and that $a$ is not the limit of critical values of $f$. The existence of {\tss}, over points other than $a$, in any neighbourhood of this {\dnlts} follows from Corollary~\ref{limitcorollary}; we need to show that in any neighbourhood of this singularity there are singularities over points other than $a$, which are either logarithmic or indirect.

The structure of the proof is as follows. We assume the contrary, and construct a sequence of {\dnltss} the projections of which have a limit. We show that this limit is itself the projection of a {\dnlts}, and use the comment after Theorem~\ref{Heins} to obtain a contradiction. Figure~1 illustrates the points and sets constructed.

\begin{figure}[ht]
	\centering
	\includegraphics[width=12cm,height=9cm]{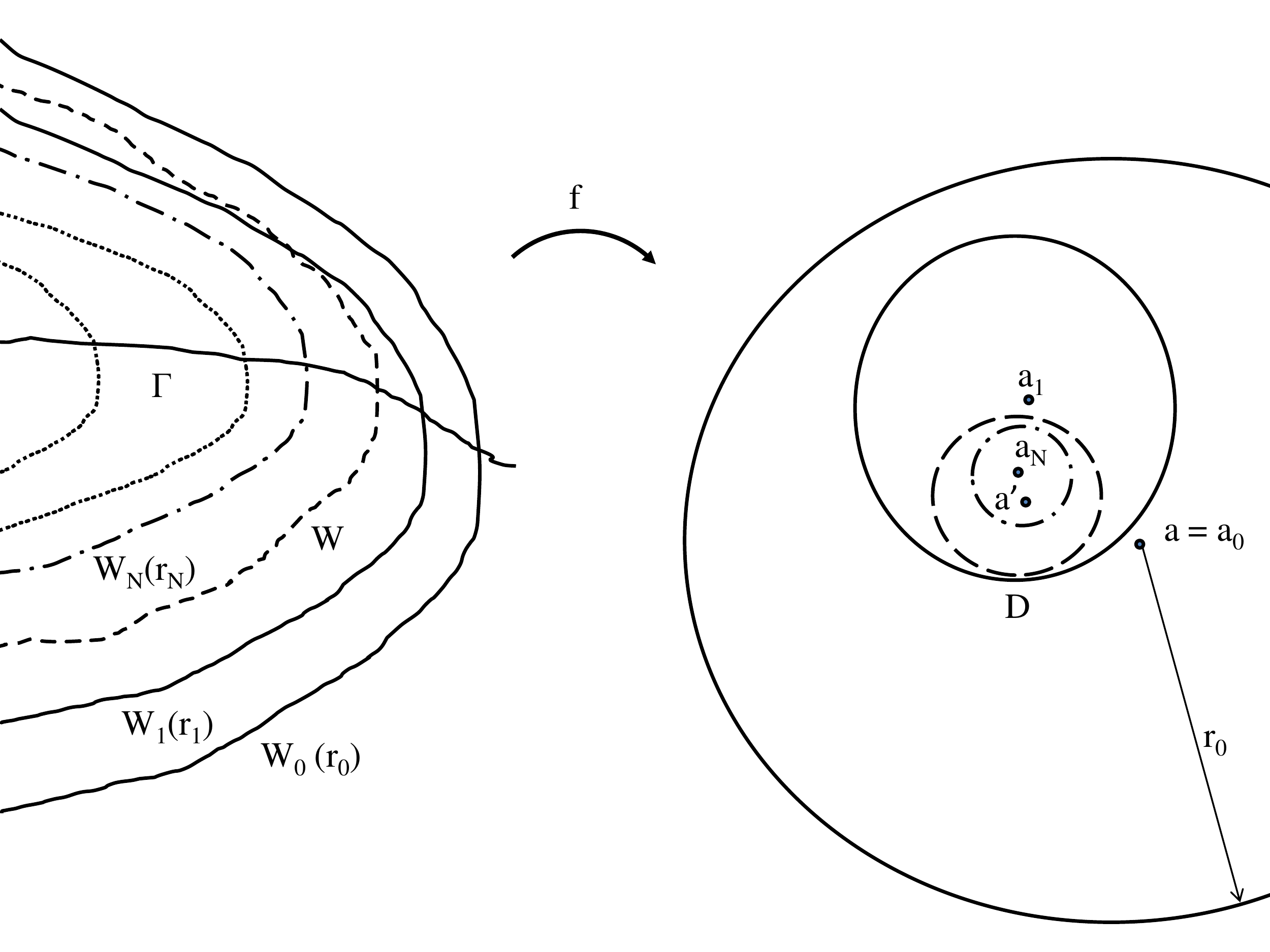}
	\caption{The construction used in the proof of Theorem~\ref{theo2}}
\end{figure}

Let $r_0>0$ be such that there are no critical values of $f$ in $B(a, r_0)$ and also, since the {\ts} is assumed to be direct, such that $f(z) \ne a$, for $z$ in the neighbourhood $U(r_0)$. Suppose also that $r_0$ is sufficiently small that all {\tss}, over points other than $a$ and with a neighbourhood contained in $U(r_0)$, are direct non-logarithmic.
 
Let $(R_n)$ be any increasing sequence of positive real numbers such that $R_n\rightarrow\infty$ as $n\rightarrow\infty$. We construct a sequence of neighbourhoods of {\dnltss} $W_n(r_n)$, with projection $a_n$ say, such that, for $n\geq 0$,
\begin{itemize}
\item $W_{n+1}(r_{n+1}) \subset W_n(r_n) \subset U(r_0)$;
\item $W_{n+1}(r_{n+1}) \cap B(0, R_{n+1}) = \emptyset$;
\item $\overline{B(a_{n+1}, r_{n+1})} \subset B(a_n, r_n)$;
\item $a_{n+1} \ne a$;
\item the equation $f(z) = a_n$ has no solutions for $z \in W_n(r_n)$;
\item $r_n \rightarrow 0$ as $n\rightarrow\infty$.
\end{itemize}

We set $W_0(r_0)=U(r_0)$, and then construct this sequence inductively.  By assumption, $W_n(r_n)$ is a neighbourhood of a {\dnlts}, so we can use Corollary~\ref{limitcorollary} to choose a {\ts} with projection $a_{n+1}$ say and with neighbourhoods $W_{n+1}(r), \ r>0,$ such that $W_{n+1}(r_{n+1}') \subset W_n(r_n)$, for some $r_{n+1}'>0$, and such that $0 < |a_{n+1} - a_n| < r_n/2$.

Next, we choose $r_{n+1}''>0$ such that $W_{n+1}(r_{n+1}'') \cap B(0, R_{n+1}) = \emptyset$. By assumption $W_{n+1}(r_{n+1}')$ is a neighbourhood of a {\dts} with projection $a_{n+1}$. Hence, there exists $r_{n+1}$ with $$0 < r_{n+1} < \min\{r_{n+1}', r_{n+1}'', |a_{n+1} - a_n|/4\}$$ such that $f(z) = a_{n+1}$ has no solutions for $z \in W_{n+1}(r_{n+1})$. Finally, both $\overline{B(a_{n+1}, r_{n+1})} \subset B(a_n, r_n)$ and $a_{n+1}\ne a$, by the choice of $a_{n+1}$ and $r_{n+1}$. This completes the construction.

Let $a' = \lim_{n\rightarrow\infty} a_n$, which exists by our choice of $r_n$. Note that, by construction, $a' \ne a$. Let $\Gamma$ be a curve produced inductively by joining a point in $W_n(r_n)$ to a point in $W_{n+1}(r_{n+1})$ using a curve lying in $W_n(r_n)$. By construction, $\Gamma$ is an asymptotic curve with asymptotic value $a'$. Hence $a'$ is the projection of a {\ts} of $f^{-1}$ which, by assumption, is direct non-logarithmic. 

Choose $r>0$ sufficiently small that $D=B(a', r) \subset B(a, r_0)$, and such that $f(z) = a'$ has no solutions in $W$, where $W$ is the component of $f^{-1}(D)$ which has unbounded intersection with $\Gamma$.

Now, by construction, $a'\in B(a_n, r_n)$, for $n\isnatural$, and so there is an $N>0$ such that $a'\in B(a_N, r_N) \subset D$ and also $a' \ne a_N$. Note that $W_N(r_N) \subset W$, since the intersection of $\Gamma$ and $W_N(r_N)$ is unbounded. Then $a'$ and $a_N$ are two distinct points in $B(a_N, r_N)$ such that $f(z) \in \{a', a_N\}$ has no solutions in $W_N(r_N)$, which is contrary to Theorem~\ref{Heins}.
\end{proof}
%
%
%
%
%
\section{The Eremenko-Lyubich Class}
\label{S2}
In this section we prove Theorem~\ref{theo1}.
%
%
%
We need the following, \cite[Theorem I.2.2]{MR1230383}.
\begin{theo}
\label{Lindelof}
Suppose that $W \subset \widehat{\mathbb{C}}$ is simply connected and $\partial W$ has more than one point. Suppose that $\psi$ maps $W$ conformally to $\disc$. Let $\Gamma$ be a Jordan arc in $W$ with endpoint $z_0 \in \partial W$. Then the curve $\psi(\Gamma)$ terminates in a point $s_0\in \partial \disc$, and $\psi^{-1}(s) \rightarrow z_0$ as $s \rightarrow s_0$ inside any Stolz angle at $s_0$.
\end{theo}
Here a \itshape Stolz angle \normalfont at $s_0\in \partial \disc$ is a set of the form; \[\{s\in\disc : | \arg(1 - \overline{s_0} s)| < \alpha, |s - s_0| < d \}, \qfor 0 < \alpha < \pi/2, \ d < 2\cos\alpha.\] 

%
%
We also need the following result, which is a version of \cite[Theorem 1]{MR1344897} that includes some assertions that appear only in the proof of that result; see also, \cite[Theorem 6.2.3]{MR2757285}.
\begin{theo}
\label{BergErem}
Suppose that $f$ is a {\tef} with an {\its} with projection $a\in\widehat{\mathbb{C}}$. Suppose that $a$ is not the limit of critical values of $f$. Then there exists a sequence of asymptotic values $(a_n)$, which converge to $a$, a sequence of disjoint unbounded simply connected domains $(U_n)$ such that $D_n~=~f(U_n)$ is a disc with $a_n\in\partial D_n$, and a sequence of asymptotic curves $(\Gamma_n)$ such that $\Gamma_n \subset U_n$, $f(\Gamma_n)$ is a radius of $D_n$ ending at $a_n$, and $f$ is univalent in $U_n$.
\end{theo}
%
%
Finally, we need the following lemma.
\begin{lemma}
\label{lem1}
Let $f$ be a {\tef}. Suppose that for every $R~>~0$ there exist $r>0$, $a_0\iscomplex$ with $|a_0|>R$, an asymptotic curve $\Gamma'$ with asymptotic value $a_0$, $W$ a simply connected neighbourhood of $\Gamma'$, and an analytic map $\phi$, univalent on $W$, such that $\phi(\Gamma')$ is an interval $(-\infty,x_0)$, and
\begin{equation}
\label{newfeq}
f(z) = re^{\phi(z)} + a_0, \qfor z\in W.
\end{equation}
Then $\eta_f = 0$.
\end{lemma}
\begin{proof}
Suppose that $\eta_f \neq 0$. Then there exist $\epsilon, R>0$ such that
\begin{equation}
\label{sisnotzero}
\left|z\frac{f'(z)}{f(z)}\right| > \epsilon, \qfor |f(z)| > R.
\end{equation}
Choose $a_0$ such that $|a_0|> 2R$, let $h = \phi^{-1}$ and put $t=\phi(z)$. Then, as $z\rightarrow\infty$ along $\Gamma'$, by (\ref{newfeq}) and (\ref{sisnotzero}),
\begin{align}
\epsilon &< \left| z \frac{\phi'(z) r e^{\phi(z)}}{re^{\phi(z)} + a_0} \right| 
          = \left| h(t) \frac{\phi'(h(t)) re^{t}}{re^{t} + a_0} \right| 
          = \left| \frac{h(t) re^{t}}{h'(t) (re^{t} + a_0)} \right| 
          \sim \left| \frac{h(t) re^{t}}{h'(t) a_0} \right|.
\end{align}
Hence, for sufficiently large negative values of $t$,
\begin{equation}
\label{ee}
\left| \frac{h'(t)}{h(t)} \right| < \frac{2re^{t}}{\epsilon |a_0|}.
\end{equation}
Without loss of generality we can assume that (\ref{ee}) applies for all $t\in (-\infty,x_0)$. Since $W$ is simply connected, and since we can assume that $0\notin W$, we can define a branch of the logarithm, $L$, in $W$. Then, by (\ref{ee}),
\begin{equation}
\label{ez}
\left| \frac{d}{dt} L(h(t)) \right| < \frac{2re^{t}}{\epsilon |a_0|}.
\end{equation}

We set $\zeta = L(h(t))$ and integrate (\ref{ez}), to obtain
\begin{equation}
\label{ineq}
\frac{2r}{\epsilon |a_0|} \int_{-\infty}^{x_0} e^{t} \ dt > \int_{-\infty}^{x_0} \left|\frac{d}{dt} L(h(t))\right| \ dt  > \left| \int_{-\infty}^{x_0} \frac{d}{dt} L(h(t)) \ dt \right| = \left| \int_{L(\Gamma')} d\zeta \right|.
\end{equation}

Now, $L(\Gamma')$ is an unbounded curve, and so the right-hand side of (\ref{ineq}) is unbounded. However, the left-hand side of (\ref{ineq}) is bounded. This contradiction completes the proof.
\end{proof}

%
%
We now prove Theorem~\ref{theo1}.
\begin{proof}[Proof of Theorem~\ref{theo1}]
As mentioned in the introduction, it is clear that if $f\in\mathcal{B}$ then $\eta_f=\infty$. Suppose, then, that $\eta_f \ne 0$. It is immediate from (\ref{def_of_s}) that the set of critical values of $f$ is bounded. To complete the proof, we show that $f$ cannot have an unbounded set of finite asymptotic values, and so $f\in\mathcal{B}$, and hence $\eta_f = \infty$. To achieve this we show first that $f$ cannot have an unbounded set of projections of {\dltss}. We then show that $f$ cannot have an unbounded set of projections of {\itss}. Finally, we show that $f$ cannot have an unbounded set of projections of {\dnltss}. 

%
%
Our first claim then is that $f$ cannot have an unbounded set of projections of {\dltss}. Suppose that, for every $R>0$, $f$ has a {\dlts} with projection $a_0\iscomplex$, such that $|a_0|~>~R$. Noting that $a_0$ is finite, we apply Corollary~\ref{dltscorollary} to obtain a simply connected neighbourhood, $W=U(r)$, of the singularity, and a conformal map $\phi: W \to \hplane$ such that (\ref{newfeq}) holds for some $r>0$. Let $\Gamma$ be an asymptotic curve in $W$ associated with the {\dlts}.

Put $t=\phi(z)$ and let \[s = \sigma(t) = \frac{1-t}{1+t}.\] Then $\psi = \sigma \circ \phi$ is a conformal mapping of $W$ to $\disc$. Moreover $\psi(\Gamma)$ is a curve in $\disc$ tending to $-1$.

We now construct another curve to $\infty$ in $W$ which satisfies the hypotheses of Lemma~\ref{lem1}. Let $\Gamma' = \phi^{-1}((-\infty, 0))$. Then $\gamma = \psi(\Gamma')$ is a curve in $\disc$ tending to $-1$ within a Stolz angle. By Theorem~\ref{Lindelof}, $\psi^{-1}(s) \rightarrow \infty$ as $s\rightarrow -1$ along $\gamma$, and so $\Gamma'$ is an asymptotic curve. Moreover, $re^{t} + a_0 \rightarrow a_0$ as $t \rightarrow -\infty$ along $\phi(\Gamma')$, and so $\Gamma'$ has asymptotic value $a_0$. A contradiction follow by Lemma~\ref{lem1}, since we have assumed that $|a_0|>R$. This establishes our initial claim. 

%
%
We next show that $f$ cannot have an unbounded set of projections of {\itss}. Suppose that, for every $R>0$, $f$ has an {\its} with projection $a\iscomplex$, such that $|a|>2R$. By Theorem~\ref{BergErem}, $f$ has an asymptotic value $a_0$, with $|a_0|>R$, an asymptotic curve $\Gamma'$ associated with $a_0$, and an unbounded simply connected domain $W$ containing $\Gamma'$ such that $f$ is univalent in $W$. Moreover, $f(W)$ is a disc, $D$, with $a_0 \in \partial D$, and $f(\Gamma')$ is a radius in $D$ ending at $a_0$.

Without loss of generality, by composing with a rotation if necessary, assume that the centre of $D$ is at $a_0 + e^{x_0}$, for some $x_0 \isreal$. Define a branch of the logarithm, $L_1$, such that $\psi(w) = L_1(w - a_0)$ is a univalent map on $D$. Let $\phi$ be the univalent map $\phi = \psi \circ f$. Note that $\phi(\Gamma') = (-\infty, x_0)$, and (\ref{newfeq}) holds with $r=1$. A contradiction follows by Lemma~\ref{lem1}, since we have assumed that $|a_0|>R$. This establishes our second claim. 

%
%
Finally we show that the projections of {\dnltss} are bounded. This follows immediately from the fact that the projections of other types of {\tss} are bounded and from Theorem~\ref{theo2}. This completes the proof.
\end{proof}
%
%
\itshape Remark: \normalfont
It seems possible to generalise the result of Theorem~\ref{theo1} to transcendental meromorphic functions with direct tracts (see, for example, \cite{MR2439666} for more background on this concept). We have not done this here, for reasons of simplicity. However, the proof seems to work similarly, although a number of results used in this paper need to be generalised. In addition, we need to replace Theorem~\ref{Heins} with \cite[Corollary 1]{MR1703869}, and \cite[Lemma 1]{MR1196102} with \cite[Lemma 6.3]{MR2439666}. \\

%
%
\itshape Acknowledgment: \normalfont
The author is grateful to Phil Rippon and Gwyneth Stallard for all their help with this paper.
%
%
%
%
%
\bibliographystyle{acm}
\bibliography{Main}
\end{document}